\documentclass[12pt]{amsart}
\usepackage{amssymb}

\numberwithin{equation}{section}
\theoremstyle{plain}
\newtheorem{theorem}{Theorem}[section]
\newtheorem{proposition}[theorem]{Proposition}
\newtheorem{cor}[theorem]{Corollary}
\newtheorem{lemma}[theorem]{Lemma}

\theoremstyle{definition}

\theoremstyle{remark}
\newtheorem{rem}{Remark}[section]

\def\R{{\mathbb R}}

\def\Cpx{{\mathbb C}}

\def\b#1{{\left\{{#1}\right\}}}

\def\n#1{{\left\|{#1}\right\|}}

\def\supp{\operatorname{supp}}

\def\Rn{{{\mathbb R}^n}}

\title{Criteria for Bochner's extension problem}

\author[]{Michael Ruzhansky and Mitsuru Sugimoto}
\address{
  Michael Ruzhansky:
  \endgraf
  Department of Mathematics
  \endgraf
  Imperial College London
  \endgraf
  180 Queen's Gate, London SW7 2AZ, UK
  \endgraf
  {\it E-mail address} {\rm m.ruzhansky@imperial.ac.uk}
  \endgraf
  \medskip
  Mitsuru Sugimoto:
  \endgraf
  Department of Mathematics, Graduate School of Science
  \endgraf
  Osaka University
  \endgraf
  Machikaneyama-cho 1-16, Toyonaka, Osaka 560-0043, Japan
  \endgraf
  {\it E-mail address} {\rm sugimoto@math.sci.osaka-u.ac.jp}
  }
\thanks{2000 MSC Classification 35B60; 35D10. \\
 \indent
The first author was supported by the JSPS Invitational
 Research Fellowship.}

\date{\today}
\begin{document}
\begin{abstract}
A necessary and sufficient condition for the resolution
of the weak extension problem is given.
This criterion is applied to also give a criterion for the
solvability of the classical Bochner's extension problem
in the $L^p$--category. The solution of the 
$L^p$--extension problem by Bochner \cite{Bo}
giving the relation between the order of the operator,
the dimension, and index $p$, for which 
the $L^p$--extension property holds, can be viewed
as a subcritical case of the general $L^p$--extension problem. 
In general, this
property fails in some critical and in all supercritical
cases. In this paper, the $L^p$--extension problem
is investigated for operators of all orders and 
for all $1\leq p\leq\infty$. Necessary and sufficient
conditions on the subset of $L^p$ are given for
which the $L^p$--extension property still holds,
in the critical and supercritical cases.
\end{abstract}

\maketitle

\section{Introduction}

In this paper we investigate the removable singularity problem
for a partial differential operator $P=P(x,D)$.
Let us assume that the open set $X\subset\Rn$ contains 
the origin
$$
A=\{0\}.
$$
We assume that the distribution $u\in\mathcal{D}'(X)$ 
satisfies equation 
\begin{equation}\label{EQ:Pwithout}
Pu=0 \textrm{ on } X\backslash A
\end{equation} 
on the punctured space $X\backslash A$, and 
the question is whether 
$u$ extends to a solution of $Pu=0$ on the whole space $X$.
We assume that operator $P$ is given by
$$
 P(x,D)=\sum_{|\alpha|\leq m} a_\alpha(x) \partial^\alpha,
$$
with complex-valued coefficients $a_\alpha(x)$.
Equalities to zero on $X$ or on
$X\backslash A$ are always understood in the distribution sense.
We will also discuss the cases when $P$ is a matrix-valued
operator and when 
$A$ is a smooth submanifold of $X$.

We will first give a criterion for the solvability of this problem.
The classical Bochner's solution gives the relation
$m\leq n(1-1/p)$, $1\leq p<\infty$, for the extension to
hold in $L^p(X)$, and the extension property is known to fail in
general for other exponents. In this paper we will first give
necessary and sufficient conditions for the solvability 
of the general weak extension problem in the distribution
space $\mathcal{D}'(X)$. Moreover, in the $L^p$--setting this will
recover Bochner's result (which can be viewed as the
subcritical case of the general $L^p$--extension problem), 
but will also yield the necessary and sufficient criterion
for the critical ($L^\infty$) and supercritical
($L^p$, $1\leq p\leq\infty$, and $m>n(1-1/p)$) cases.

The main idea
is to consider the asymptotic behaviour of the quantity
\[
||Q(aPu*\phi_\epsilon)||_{L^r}
\]
as $\epsilon>0$ tends to $0$, where
$a\in C_0^\infty(X)$, $\phi\in C_0^\infty(\R^n)$,
$\phi_\epsilon(x)=\epsilon^{-n}\phi(x/\epsilon)$,
and $Q$ is a partial differential operator with constant coefficients.
The criterion can be described by using the order of its behaviour
in $\epsilon$.
The precise statement of this main result (Theorem \ref{cor:all-U-Lp})
will be given in Section \ref{Main}.

As an application of this criterion,
we can cover the related known and prove new results on
the $L^p$--extension problem, namely
we will assume that solution $u$ to \eqref{EQ:Pwithout}
satisfies $u\in L^p(X)$, $1\leq p\leq\infty$.
We will say that the {\em $L^p$--extension holds}
if equality \eqref{EQ:Pwithout} with $u\in L^p(X)$
implies that $Pu=0$ on $X$. Sometimes we will need to
specify $u$ in this implication, in which case we will
say that the $L^p$--extension holds for $u$.
Conversely, we will say that 
the {\em $L^p$--extension fails} if there is $u\in L^p(X)$
such that \eqref{EQ:Pwithout} holds but $Pu\not=0$ on $X$.
The following result
goes back to Bochner \cite{Bo}, and implies that the
$L^p$--extension holds for 
$m\leq n(1-1/p)$ and $1\leq p<\infty$.

\begin{theorem}\label{Th-A}
  Let $u\in L^p(X)$, $m\leq n(1-1/p)$,
  $1\leq p<\infty$, and $Pu=0$ on $X\backslash A$.
  Then $Pu=0$ on $X$.
\end{theorem}

Theorem \ref{Th-A} is also proved by another method
of Sugimoto and Uchida \cite{SU}.
For a brief explanation of the methods
used in \cite{Bo} and \cite{SU}, see also
\cite{S1} and \cite{S2}.
But we can obtain the same results from our criterion as well,
see Corollary \ref{cor:FS-stronger} and the proof.
A striking feature of this result is that it is independent
of the type of operator $P$, and that its proof is rather simple.
Various generalisations of
this relation have been studied over the years describing
conditions on more general removable sets in terms of
capacities and other quantities (see e.g. 
\cite{Bo}, \cite{HP}, \cite{Li}, \cite{P}, \cite{Po},
\cite{SU}, and references therein, 
to mention only a few). Moreover, most of the existing
literature deals with the subcritical and certain critical
cases of the weak extension problem (in the terminology
introduced below in detail), while the present paper
concentrates on the general critical and supercritical
cases (while recovering subcritical results as well).

In view of Theorem \ref{Th-A}, for all $1\leq p\leq\infty$,
we will call the cases $m<n(1-1/p)$, $m=n(1-1/p)$,
and $m>n(1-1/p)$ to be {\em subcritical, critical}, and 
{\em supercritical}, respectively. Thus, Bochner's theorem
says that for $1\leq p<\infty$ the $L^p$--extension
property holds in the subcritical and critical cases.
We remark that a generalisation of Bochner's
Theorem \ref{Th-A} to the supercritical case
is not generally possible because
the inequality $m\leq n(1-1/p)$ in Theorem \ref{Th-A} 
is sharp. Indeed, it is shown by e.g. John \cite{Jo}
that
if $P$ has analytic coefficients and
is elliptic, then function
\begin{equation}\label{EQ:John}
u(x)=|x|^{m-n}\left\{A(x)+B(x)\log |x|\right\}
\end{equation} 
solves $Pu=\delta$ on $X$, with some functions
$A(x), B(x)$ bounded in a neighbourhood of $0\in X$.
Thus, $Pu=0$ on $X\backslash A$, and we can remark
that $u\in L^p(X)$ for $m>n(1-1/p)$, implying that the
$L^p$--extension property fails for $1\leq p<\infty$
in the supercritical cases.

Theorem \ref{cor:all-U-Lp} clearly also gives necessary
and sufficient conditions when $u\in L^p(X)$,
for all $1\leq p\leq\infty$ and all orders $m$ of operator $P$.
Moreover, this also yields some further simple sufficient conditions
since the quantities in \eqref{EQ:cond-gamma-cora} and
\eqref{EQ:cond-gamma-cor} can be efficiently estimated
for $u\in L^p(X)$. Therefore, our criterion 
will also give sufficient 
conditions on functions for which the $L^p$--extension
property still holds in the supercritical cases
when the $L^p$--extension
property fails in general.
All of such results will be given in Section \ref{Lpext}.

We will also consider the critical case $p=\infty$ 
(especially when $m=n$)
which is excluded from Theorem \ref{Th-A}.
Considering this case has
special meaning because it has many important application.
However, the proof of
Bochner \cite{Bo} breaks down in the case $p=\infty$ and
$m\geq n$.
Moreover, the alternative proof of Theorem \ref{Th-A} given by
Sugimoto and Uchida \cite{SU} 
also breaks down in view of the failure of the boundedness of
pseudo-differential operators of order zero on the space
$L^\infty(X)$.
On the other hand, example
\eqref{EQ:John} shows that the $L^\infty$--extension  
fails for $m>n$.
Thus, the most subtle critical
case is $m=n$. In fact, various capacity methods
(see e.g. Littman \cite{Li}, Harvey and Polking
\cite{HP}, Pokrovskii \cite{P}, or Polking \cite{Po})
are not applicable to this situation in view of 
simple examples given in this paper.
Our criterion will enable us to
carry out a rather comprehensive analysis in this case.
Such discussion will be given in Section \ref{Linfty}.

In Section \ref{Sec:generalisations} we will give several 
generalisations of
the obtained results. In particular, we will discuss
modifications of statement in the case when the set $A$
is not a point, but a smooth submanifold of $X$ of
codimension $k$. In this case we can factor out the
problem to reduce the analysis to the previous situations.
In particular, this yields another proof of the fact that if
$A$ is a (complex) hypersurface in $\Cpx^n$ and 
if a bounded function
is analytic in $\Cpx^n\backslash A$, then it is analytic
in $\Cpx^n$ (in the case $n=1$ this is the celebrated Riemann's
extension theorem, since $A$ is a point). Among other things, in 
Section \ref{Sec:generalisations}
we will also discuss the case when $A$ has a 
real codimension one in $\Cpx^n$, which can be viewed
as the critical $L^\infty$--extension problem.

\section{A criterion for weak extension}
\label{Main}

Let $X\subset\R^n$ be an open set which contains the origin $A=\b{0}$.
Let operator $P$ be a partial differential operator on $X$ of order $m$,
which always means in the paper that it is of the form
\begin{equation}\label{EQ:P-form}
 P=\sum_{|\alpha|\leq m} a_\alpha(x) \partial^\alpha,
\quad \text{$a_\alpha\in C^\infty(X)$}.
\end{equation}
Functions $\b{a_\alpha(x)}_{|\alpha|\leq m}$ may be complex-valued.
If they are constants, 
we call $P$ a partial differential equation
with constant coefficients of order $m$. The condition
of the smoothness of $a_\alpha$ can be weakened in
some cases, see e.g. Remark \ref{nonsmooth} and also
Section \ref{Lpext}.

For the distribution $u\in \mathcal{D}'(X)$ on $X$,
the distribution $Pu\in \mathcal{D}'(X)$ 
is defined
by
\[
\langle Pu,\,\phi\rangle
=\sum_{|\alpha|\leq m}
(-1)^{|\alpha|}\langle u,\,\partial^\alpha(a_\alpha\phi)\rangle
\]
for $\phi\in C_0^\infty(X)$.
It can be also regarded as a distribution on a smaller open set
$X\backslash A$ if we take test functions $\phi$
in $C_0^\infty(X\backslash A)$.
We note that $Pu=0$ on $X\backslash A$
does not always imply $Pu=0$ on $X$,
but we will give a criterion on such weak extension property.
Since the problem is local, we may multiply $P$ by
a cut-off function $a\in C^\infty_0(X)$ at the origin $A$.
Then we can regard $aPu$ as a distribution on $\R^n$.
The idea for the analysis is to consider convolutions of
$a Pu$ with certain families of functions, where
the convolution of a distribution $f$ on $\R^n$ with a 
smooth function  $\phi\in C_0^\infty(\R^n)$
is defined in the usual way by
$$
 (f*\phi)(x)=\langle f,\,\phi(x-\cdot)\rangle. 
$$
For $\phi\in C_0^\infty(\R^n)$ and $\epsilon>0$, we will always denote
\begin{equation}\label{molifier}
\phi_\epsilon(x)=\epsilon^{-n}\phi(x/\epsilon).
\end{equation}
The following is a necessary and sufficient condition for
the weak extension property to hold in the space of
distributions. It will give many application to, in particular,
the $L^p$--extension problem in various settings.

\begin{theorem}\label{cor:all-U-Lp}
Let $X\subset\R^n$ be an open set, let
$A=\b{0}$ be the origin, 
and let
$P$ be a partial differential operator on $X$ 
as in \eqref{EQ:P-form}.
Suppose that $u\in \mathcal{D}'(X)$ satisfies
$$ Pu=0 \textrm{ on } X\backslash A.$$
Then the following statements are equivalent:
\begin{itemize}
\item[(i)] $Pu=0$ on $X$;
\item[(ii)] There
exist $0<r\leq\infty$, $a\in C_0^\infty(X)$
which is non-zero at the origin $A$,
non-zero $\phi\in C_0^\infty(\R^n)$, a sequence
$\epsilon_j\to 0+$, and
a non-zero partial differential operator $Q$
of order $d$ with constant coefficients such that
\begin{equation}\label{EQ:cond-gamma-cora}
 \epsilon_j^{n(1-1/r)+d}||
 Q(aPu*\phi_{\epsilon_j})
 ||_{L^r(\R^n)}\to 0 \textrm{ as } \epsilon_j\to 0+;
\end{equation} 
\item[(iii)] For every $0<r\leq\infty$,
all functions $a\in C_0^\infty(X)$,
$\phi\in C_0^\infty(\R^n)$,
and every partial differential operator $Q$
of order $d$ with constant coefficients we have
\begin{equation}\label{EQ:cond-gamma-cor}
 \epsilon^{n(1-1/r)+d}||
 Q(aPu*\phi_{\epsilon})
 ||_{L^r(\R^n)}\to 0 \textrm{ as } \epsilon\to 0+.
\end{equation} 
\end{itemize}
\end{theorem}
We remark that $Q(aPu*\phi_{\epsilon})$ always tends to
$\widehat \phi(0)Q(aPu)$ as $\epsilon\to 0+$ in 
the distributional sense, so that the expression
$$ \epsilon^{n(1-1/r)+d}
 Q(aPu*\phi_{\epsilon})
$$
always tends to zero as $\epsilon\to 0+$
in the the distributional sense when $d>-n(1-1/r)$, 
in particular when $1<r\leq\infty$.
Property \eqref{EQ:cond-gamma-cora}
or \eqref{EQ:cond-gamma-cor} requires that in order for the
weak extension property to hold,
this convergence should be in a stronger sense. See also Remark
\ref{rem:weakconv} for some special case.
\begin{rem}\label{nonsmooth}
We can assume that the operator $P$ has non-smooth coefficients
if we restrict the class of distributions for $u$.
For example, if the coefficients $\b{a_\alpha}_{|\alpha|\leq m}$ in 
\eqref{EQ:P-form} satisfy only $a_\alpha\in C^{|\alpha|+k}$ but
the distribution $u$ is of order $k$ with some 
non-negative integer $k$, then 
$\langle Pu,\phi\rangle$ is still well-defined
and Theorem \ref{cor:all-U-Lp} is still valid in this setting.
\end{rem}

\begin{proof}[Proof of Theorem \ref{cor:all-U-Lp}.]
Implications (i) $\Longrightarrow$ (iii) and (iii) 
$\Longrightarrow$ (ii) are straightforward,
while the implication (ii) $\Longrightarrow$ (i)
is based on the structure theorem for distributions.
Indeed, assumption
$Pu=0$ on $X\backslash A$
implies that the support of 
distribution $Pu$ on $X$ is contained in the origin $A=\{0\}$, so
we get that
$Pu$ is a finite sum of the derivatives of Dirac's delta function.
By multiplying a cut-off function $a\in C^\infty_0(X)$,
we have the expression
\begin{equation}\label{str-thm}
aPu=S\delta\text{ on $\R^n$},
\end{equation}
for some partial differential operator $S$ with
constant coefficients.
This conclusion is also true in the setting of Remark \ref{nonsmooth}.
All our task now is to show that $S=0$, 
so that we can conclude that $Pu=0$ on $X$.

We consider the convolution of both hand sides of
\eqref{str-thm} with test functions, and obtain the equality
\begin{equation}\label{str-thm:conv}
aPu*\phi_\epsilon=S\phi_\epsilon,
\end{equation}
where $\phi_\epsilon$ is defined as in \eqref{molifier}
for $\phi\in C^\infty_0(\R^n)$.
Implication (ii) $\Longrightarrow$ (i)
readily follows from the following lemma:

\begin{lemma}\label{Lemma:2}
Let $0< r\leq\infty$.
Let $S$ be a non-zero partial differential operator of
order $\kappa$ with constant coefficients, and assume that
its top $\kappa^{th}$ order part is non-zero.
Let $\phi\in C_0^\infty(\R^n)$ be a non-zero function.
Then we have
\begin{equation}\label{EQ:conv4s}
C\epsilon^{-\kappa-n(1-1/r)}\leq
||
 S\phi_{\epsilon}
 ||_{L^r(\R^n)}
 +O(\epsilon^{-\kappa-n(1-1/r)+1})
\end{equation}
for $0<\epsilon\leq1$, where $C>0$ is a constant
independent of $\epsilon$.
\end{lemma}
Let us first show that Lemma \ref{Lemma:2} implies that $S=0$.
In fact, on account of the expression \eqref{str-thm:conv},
condition (ii) of Theorem \ref{cor:all-U-Lp} says that
\[
 \epsilon_j^{n(1-1/r)+d}||
 QS\phi_{\epsilon_j}
 ||_{L^r(\R^n)}\to 0,
\] 
while we obtain from Lemma \ref{Lemma:2} that
\[ 
0<C\leq\epsilon_j^{n(1-1/r)+d+\kappa}
||
 QS\phi_{\epsilon_j}
 ||_{L^r(\R^n)}
 +O(\epsilon_j)
\]
as $\epsilon_j\to 0+$, where $d$ and $\kappa$ are the orders of 
$Q$ and $S$, respectively.
It is a contradiction if $\kappa\geq0$, hence we must have $S=0$.
\end{proof}

\begin{proof}[Proof of Lemma \ref{Lemma:2}.]
Since the order of the partial differential operator $S$ is
$\kappa$, we can write it as $S=\sum_{|\alpha|\leq \kappa} c_\alpha
\partial^\alpha.$
Denoting its homogeneous top order part by
$S_\kappa=\sum_{|\alpha|=\kappa} c_\alpha \partial^\alpha$,
we have
\begin{align*}
S\phi_\epsilon(x)
&=\epsilon^{-n}
\sum_{|\alpha|\leq \kappa} \epsilon^{-|\alpha|} c_\alpha
(\partial^\alpha\phi)(x/\epsilon)
\\
&=\epsilon^{-\kappa-n}
(S_\kappa\phi)(x/\epsilon)+F_\epsilon(x/\epsilon),
\end{align*}
where
$$
F_\epsilon(x)=\sum_{|\alpha|\leq \kappa-1}
\epsilon^{\kappa-|\alpha|}
c_\alpha (\partial^\alpha \phi)(x)\quad\text{and}\quad
||F_\epsilon(x)||_{L^r}=O(\epsilon).
$$
Taking the $L^r$ quantities of both sides, we have
\[
\epsilon^{-\kappa-n(1-1/r)}||S_\kappa\phi||_{L^r}
\leq
C||
 S\phi_{\epsilon}
 ||_{L^r(\R^n)}
 +O(\epsilon^{-\kappa-n(1-1/r)+1}),
\]
Then the proof is complete if we show
that $||S_\kappa\phi||_{L^r}\not=0$. 
Indeed, if $S_\kappa\phi\equiv 0$,
then taking its Fourier transform, we would get
$S_\kappa(\xi)\widehat{\phi}(\xi)=0$ for all $\xi\in\Rn$.
Since $S_\kappa(\xi)$ is a homogeneous polynomial of degree $\kappa$
and $S_\kappa\not\equiv0$,
its zero set is nowhere dense in $\Rn$. Since
$\widehat{\phi}$ is smooth, we get that $\widehat{\phi}(\xi)=0$
for all $\xi$, which contradicts the assumption that
$\phi\not\equiv 0$.
\end{proof}

\section{$L^p$-extension problem}
\label{Lpext}
In this section, we apply Theorem \ref{cor:all-U-Lp}
for the distribution $u\in L^p(X)$.
Since it is of order $0$,
we may assume that $P$ in \eqref{EQ:P-form} is of the weaker form
\begin{equation}\label{EQ:P-formw}
 P=\sum_{|\alpha|\leq m} a_\alpha(x) \partial^\alpha,
\quad \text{$a_\alpha\in C^{|\alpha|}(X)$}.
\end{equation}
as is mentioned in Remark \ref{nonsmooth}.
Using the criterion of Theorem \ref{cor:all-U-Lp},
we can now derive good estimates for the behaviour of the
quantity
 $||Q(aPu*\phi_{\epsilon})||_{L^r(\R^n)}$ as $\epsilon\to 0+$,
which appeared in Theorem \ref{cor:all-U-Lp}.

\begin{proposition}\label{cor:all-U-Lp2}
Let $X\subset\R^n$ be an open set, let
$A=\b{0}$ be the origin, 
and let
$P$ be a partial differential operator on $X$ of
order $m$.
Let $a\in C_0^\infty(X)$, $\phi\in C_0^\infty(\R^n)$,
and let $Q$ be a partial differential operator
of order $d$ with constant coefficients.
Let $1\leq p\leq r\leq\infty$ and suppose $u\in L^p(X)$.
Then we have 
\begin{equation}\label{O-beh}
 ||
 Q(aPu*\phi_{\epsilon})
 ||_{L^r(\R^n)}
=O(\epsilon^{-d-m-n(1/p-1/r)}) \textrm{ as } \epsilon\to 0+.
\end{equation}
Furthermore, if $p\neq\infty$ and $Pu=0$ on $X\backslash A$,
then we have
\begin{equation}\label{o-beh}
 ||
  Q(aPu*\phi_{\epsilon})
 ||_{L^r(\R^n)}
=o(\epsilon^{-d-m-n(1/p-1/r)}) \textrm{ as } \epsilon\to 0+.
\end{equation}
\end{proposition}
The proof of Proposition \ref{cor:all-U-Lp2} will be given later in this
section, but now let us state some results obtained from it
as corollaries of Theorem \ref{cor:all-U-Lp}.
First we note that Bochner's result (Theorem
\ref{Th-A} in Introduction) or even its generalised version
stated below is a straightforward consequence of
Theorem \ref{cor:all-U-Lp} and Proposition \ref{cor:all-U-Lp2}
with $r=p$.

\begin{cor}\label{cor:FS-stronger}
Let $1\leq p\leq\infty$, let $X\subset\Rn$ be an open set,
let $A=\{0\}$ be the origin,
and let
$P$ be a partial differential operator on $X$ 
as in \eqref{EQ:P-formw} of
order $m<n(1-1/p)$.
Suppose that $u\in L^p(X)$ satisfies
$Pu=0 \textrm{ on } X\backslash A$.
Then $Pu=0$ on $X$.
Moreover, if $1\leq p<\infty$, the conclusion
holds also for $m\leq n(1-1/p)$.
\end{cor}

While Theorem \ref{cor:all-U-Lp} gives the necessary and sufficient
condition also in the $L^p$--setting, 
the combination of Theorem \ref{cor:all-U-Lp} and
Proposition \ref{cor:all-U-Lp2}
can also give sufficient conditions for the $L^p$-extension property
for the critical case $m=n(1-1/p)$ with $p=\infty$,
or even for the supercritical case $m>n(1-1/p)$
for general $1\leq p\leq\infty$.
Indeed we have the following result:

\begin{cor}\label{Prop:harm}
Let $1\leq p\leq\infty$,
let $X\subset\Rn$ be an open set, 
let $A=\{0\}$ be the origin, and let $P$ be 
a partial differential operator of
order $m<n(1-1/p)+1$ with smooth coefficients 
as in \eqref{EQ:P-form}.
Suppose that $u\in L^p(X)$ satisfies
$Pu=0 \textrm{ on } X\backslash A$.
Let $1\leq q\leq\infty$, let $Q$ be a non-zero partial differential
operator of order $d>m-n(1-1/q)$ with constant coefficients, and
also suppose that $Qu\in L^q$.
Then $Pu=0$ on $X$.
Moreover, if $1\leq p<\infty$, the conclusion
holds also for $m\leq n(1-1/p)+1$.
If $1\leq q<\infty$, the conclusion
holds also for $d\geq m-n(1-1/q)$.
\end{cor}

\begin{rem}\label{remark2}
As a special case of the assumption $Qu\in L^q$
in Corollary \ref{Prop:harm}, we may assume that $u$
belongs to the Sobolev space $W^{q,d}$.
More strongly, we may just assume $Qu=0$ with a non-zero partial differential
operator $Q$.
In this case the condition $d>m-n(1-1/q)$ is automatically satisfied if
we take large $d$ as the order of $Q$.
\end{rem}

\begin{proof}[Proof of Corollary \ref{Prop:harm}.]
Let us show that property \eqref{EQ:cond-gamma-cor}
in Theorem \ref{cor:all-U-Lp} holds
for $Q$ which appeared in the assumption.
Indeed, we have 
$QaPu=\tilde a Ru +aPv$, where $\tilde a\in C^\infty_0(X)$
is identically one on $\supp a$,
$R=[Q,\,aP]$ is a partial differential operator of order $m+d-1$,
and $v=Qu\in L^q$. Hence we have
$Q(aPu*\phi_\epsilon)=\tilde a Ru*\phi_\epsilon+a Pv*\phi_\epsilon$.
Let $r$ be such that $p\leq r$ and $q\leq r$.
Then from Proposition \ref{cor:all-U-Lp2} with $d=0$, we obtain
\[
 ||
 \tilde aRu*\phi_\epsilon
 ||_{L^r(\R^n)}
=O(\epsilon^{-m-d+1-n(1/p-1/r)}) \textrm{ as } \epsilon\to 0+,
\]
or small $o(\epsilon^{-m-d+1-n(1/p-1/r)}))$ if $1\leq p<\infty$, and 
\[
 ||
  aPv*\phi_\epsilon
 ||_{L^r(\R^n)}
=O(\epsilon^{-m-n(1/q-1/r)}) \textrm{ as } \epsilon\to 0+,
\]
or small $o(\epsilon^{-m-n(1/q-1/r)})$ if $1\leq q<\infty$.
Then we have
\[
\epsilon^{n(1-1/r)+d}||Q(aPu*\phi_\epsilon)||_{L^r(\R^n)}
=O(\epsilon^{-m+1+n(1-1/p)})+O(\epsilon^{d-m+n(1-1/q)})
\]
or small $o(\cdot)$'s, respectively,
and the the conclusion is straightforward.
\end{proof}

Now we prove Proposition \ref{cor:all-U-Lp2}.
We note 
\begin{align*}
 (aPu*\phi_\epsilon)(x) & = \int\sum_{|\alpha|\leq m}(aa_\alpha)(x-y)
   (\partial^\alpha u)(x-y) \phi_\epsilon(y) dy \\
   & = \sum_{|\alpha|\leq m} \int \left[ (-\partial)^\alpha u(x-\cdot)
   \right](y) (aa_\alpha)(x-y) \phi_\epsilon(y) dy \\
   & = \int u(x-y) \sum_{|\alpha|\leq m} \partial_y^\alpha
   \left[ (aa_\alpha)(x-y) \phi_\epsilon(y) \right] dy, 
\end{align*}
where we used the identity
$(\partial^\alpha u)(x-y)=((-\partial)^\alpha(u(x-\cdot)))(y).$
Hence we obtain
\begin{equation}\label{EQ:P-adjoint}
  (aPu*\phi_\epsilon)(x) = (u*R_x\phi_\epsilon)(x),
\end{equation} 
where 
\begin{equation}\label{EQ:P-adj-def}
 (R_x\phi_\epsilon)(y) = \sum_{|\alpha|\leq m} \partial_y^\alpha
 \left[(aa_\alpha)(x-y)\phi_\epsilon(y)\right].
\end{equation} 
These formulae can be easily extended in the sense of
distributions for $u\in L^p(X)$ by extending by zero outside of $X\subset\R^n$.
Estimate \eqref{O-beh} is a consequence of the following lemma
if we note that $Q(aPu*\phi_\epsilon)=aPu*Q\phi_\epsilon$ and
$Q\phi_\epsilon$ is a linear combination of the functions of the form
$\epsilon^{-|\alpha|}(\partial^{\alpha}\phi)_\epsilon$
($|\alpha|\leq d$):

\begin{lemma}\label{Lemma:1}
Let $X\subset\Rn$ be an open set,
and let $P$ be a partial differential operator of
order $m$.
Let $a\in C^\infty_0(X)$ and
$\phi\in C_0^\infty(\R^n)$. 
Let $1\leq p\leq r\leq\infty$ and suppose $u\in L^p(X)$.
Then we have the estimate
\begin{equation}\label{EQ:conv4}
||
 aPu*\phi_{\epsilon}
 ||_{L^r(\R^n)}
 \leq C\epsilon^{-m-n(1/p-1/r)}||u||_{L^p(X)}
\end{equation}
for $0<\epsilon\leq1$, where $C>0$ is a constant
independent of $u$ and $\epsilon$.
\end{lemma}

\begin{proof}[Proof of Lemma \ref{Lemma:1}.]
On account of \eqref{EQ:P-adj-def},
let us define
$$v(x,z)=\sum_{|\alpha|\leq m} \epsilon^{m-|\alpha|}
\partial_z^\alpha\left[ (aa_\alpha)(x-\epsilon z)\phi(z)\right],
$$
so $v$ also depends on $\epsilon>0$.
Taking $z=\epsilon^{-1}y$, so that $\partial_z^\alpha=
\epsilon^{|\alpha|}\partial_y^\alpha$, we get
\[
 \epsilon^{-m} v(x,y/\epsilon)=
 \sum_{|\alpha|\leq m}
 \partial_y^\alpha\left[(aa_\alpha)(x-y)\phi(y/\epsilon)
 \right]=\epsilon^{n}(R_x\phi_\epsilon)(y).
\]
Denoting
\begin{equation}\label{EQ:conv5}
v_{x,\epsilon}(y)=\epsilon^{-n}v(x,y/\epsilon)=
\epsilon^m (R_x\phi_\epsilon)(y),
\end{equation}
we get by \eqref{EQ:P-adjoint} that
\[
(aPu*\phi_{\epsilon})(x)=
\epsilon^{-m}u*v_{x,\epsilon}(x).
\]
Taking the $L^r$--norm
of this equality implies that
\[
||aPu*\phi_{\epsilon}||_{L^r(\R^n)}  \leq 
\epsilon^{-m}||u||_{L^p(X)}\n{\sup_{x\in \R^n}|v_{x,\epsilon}|}_{L^q}
\]
by Young's inequality, where $1+1/r=1/p+1/q$.
Furthermore, there is a constant $C>0$ such that the estimate
$$
|v(x,z)|\leq C\sum_{|\alpha|\leq m}|(\partial^\alpha \phi)(z)|
$$
holds for all $x\in \R^n$, $z\in\R^n$, and $0<\epsilon\leq 1$,
because $aa_\alpha\in C^{|\alpha|}_0(X)$.
From this inequality and \eqref{EQ:conv5}, we obtain easily
$$
\n{\sup_{x\in \R^n}|v_{x,\epsilon}|}_{L^q}
\leq C\epsilon^{-n(1-1/q)}\sum_{|\alpha|\leq m}||\partial^\alpha \phi||_{L^q}
$$
for $0<\epsilon\leq 1$.
By all of these arguments together with estimate \eqref{EQ:conv5},
the proof is complete.
\end{proof}

If we assume $Pu=0$ on $X\backslash A$,
we have the expression \eqref{str-thm} by the structure theorem
for distributions, hence the equality
$Q(aPu*\phi_\epsilon)=QS\phi_\epsilon$.
Then by a direct computation of $||QS\phi_\epsilon||_{L^r}$,
we have readily the property
\[
||Q(aPu*\phi_\epsilon)||_{L^r(\R^n)}=O(\epsilon^{-d-\kappa-n(1-1/r)}),
\]
where $\kappa$ denotes the order of $S$.
We obtain estimate \eqref{o-beh}
if we combine this property with the following result
by Harvey and Polking \cite[Theorem 6.1]{HP},
for which we can also give a short alternative proof based 
on Lemmas \ref{Lemma:2} and \ref{Lemma:1}.

\begin{lemma}\label{Lemma:basic}
Let $1\leq p\leq\infty$ and let $X\subset\Rn$ be an open set.
Let $P$ be a partial differential operator on $X$ of
order $m$ and let $S$ be a 
partial differential operator of order $\kappa$
with constant coefficients. 
Suppose that $u\in L^p(X)$ satisfies
\begin{equation}\label{EQ:structure}
 Pu=S\delta\text{ on $X$}.
\end{equation}
Then we have $\kappa\leq m-n(1-1/p)$.
Moreover, if $1\leq p<\infty$,
then we have $\kappa<m-n(1-1/p)$.
\end{lemma}

\begin{proof}[Proof of Lemma \ref{Lemma:basic}.]
We start from the expression \eqref{str-thm:conv} since
assumption \eqref{EQ:structure} implies it.
We may assume the top $\kappa^{th}$ order of $S$ is non-zero.
Then the combination of Lemmas \ref{Lemma:2} and \ref{Lemma:1} with $r=p$
yields the estimate
\begin{equation}\label{EQ:conv41}
1 \leq C\epsilon^{\kappa-(m-n(1-1/p))}||u||_{L^p}+O(\epsilon),
\end{equation}
for $0<\epsilon\leq1$, where $C>0$ is a constant
independent of $u$ and $\epsilon$.
Hence $\kappa> m-n(1-1/p)$ implies the contradiction.
Furthermore, let us fix
$\chi\in C_0^\infty(X\cap B(0,1))$ such that $\chi\equiv 1$ in some
neighbourhood of the origin.
Then function
$u_R(x)=\chi(x/R)u(x)$ satisfies $u_R\in L^p(X)$ and
$Pu_R=S\delta$ on $X$, for sufficiently small $R>0$.
Since
$||u_R||_{L^p(X)}\leq ||u||_{L^p(|x|\leq R)}\to 0$
as $R\to 0$ in the case $1\leq p<\infty$,
plugging $u_R$ instead of $u$ in estimate
\eqref{EQ:conv41} with $\kappa=m-n(1-1/p)$
yields a contradiction again.
\end{proof}

\section{$L^\infty$-extension problem}
\label{Linfty}

In the last section, we have established Proposition \ref{cor:all-U-Lp2}
by using Lemma \ref{Lemma:basic}.
In the statements of the $L^p$--extension results which follow it,
we have sometimes only weaker results for the case $p=\infty$ than
those for the case $1\leq p<\infty$.
For example, Corollary \ref{cor:FS-stronger} says that
$m=n(1-1/p)$ is the critical index for the $L^p$-extension,
and the $L^p$--extension property actually holds if $p\neq\infty$,
while the critical case $m=n$ for $p=\infty$ is excluded.

But this exceptional case is important since it contains several
interesting situations like the analytic extension
problem, harmonic functions in $\R^2$,
or the case of bounded eigenfunctions or bounded
fundamental solutions for partial differential operators.
It is also related to the Riemann's extension theorem,
which ensures
that if a function is analytic on $\Cpx\backslash\{0\}$ and 
is bounded
near $0$, then it can be extended to an analytic function
on $\Cpx$.
This problem will be discussed again in 
Section \ref{Sec:generalisations}.

Now, we can observe that the $L^\infty$--extension fails
in general in the critical case $m=n$. 
Indeed, let $H_j$ denote the Heaviside
function $H_j(x)=0$ for $x_j<0$ and $H_j(x)=1$ for 
$x_j\geq 0$. Then we easily see that
$\partial_{x_1}H_1=\delta_0$ and
$\partial_{x_1}\partial_{x_2} (H_1 H_2)=\delta_{x_1=0}
\delta_{x_2=0}=\delta_{(0,0)}$
are examples of the failure of the $L^\infty$--extension
in $\R^1$ and $\R^2$, respectively, for operators of
order $m=n$, with similar examples in higher dimensions. 

On the other hand, it can be shown that the 
$L^\infty$--extension holds for the Laplacian $P=\Delta$
in $\R^2$. Indeed, let
$X$ be a ball in $\R^2$ of radius one, and
suppose that $\Delta u=0$ in $X\backslash A$. Then,
using spherical harmonics,
$u$ can be written in the form
$$
 u(x)=h(x)+c\log |x|,
$$    
for some constant $c$ and some function $h$ satisfying
$\Delta h=0$ in $X$. It follows that
$\Delta u=2\pi c\delta$ on $X$. On the other
hand, for the $L^\infty$--extension problem we
assume that $u\in L^\infty(X)$, thus implying that
$c=0$ and, therefore, $\Delta u=0$ in $X$.

The $L^\infty$--extension problem is closely related to
the analysis of bounded eigenfunctions of operators.
For example, let $u\in L^\infty(X)$ be a bounded function
which is an eigenfunction for the operator $P$ on a 
punctured domain, with an eigenvalue $\lambda\in\Cpx$,
i.e. let $u$ satisfy
$$
 Pu=\lambda u \textrm{ on } X\backslash A.
$$
Then $u$ is an eigenfunction of operator $P$ on the whole
domain $X$ if and only if the $L^\infty$--extension holds
for the operator $P(x,D)-\lambda$.

Thus, we will be interested in characterising operators
$P$ of order $m=n$, for which the $L^\infty$--extension
holds. We will show that in this case the 
$L^\infty$--extension holds if and only if the operator
$P$ has no bounded fundamental solutions. In general, operator
$P$ may have fundamental solutions that are bounded or
unbounded. If $P$ has constant coefficients, it is shown
in H\"ormander \cite[Theorem 3.1.1]{Ho} that $P$ always
has a fundamental solution in the space
${\mathcal B}_{\infty,\widetilde{P}}^{loc}(\Rn)$, which
roughly means that its Fourier transform is bounded
with the weight $\widetilde{P}$ (here
we use the notation of \cite{Ho}).
For example, if $P(D)$ is an elliptic
partial differential operator with constant coefficients
of order $m=n$, then its fundamental solution is given 
by the formula going back to G. Herglotz (see e.g.
Shimakura \cite{Sh}):
$$
  E(x)= \frac{i^{-2}}{(2\pi i)^n} \int_{{\mathbb S}^{n-1}}
  \frac{\log (-i\omega\cdot x)}{P(\omega)} dS_w.
$$
This function can be written as a sum of a bounded function
and a function of at most logarithmic growth
(see also \eqref{EQ:John}).
In fact, it can be shown that 
this fundamental solution is bounded for odd $n$
and has at most logarithmic growth at the origin for even $n$.
Thus, for even $n$, we need the logarithmic term to vanish
for the corresponding $P$ to have the $L^\infty$--extension 
property for all $u$. This extends the argument given above
for the Laplacian on $\R^2$.

Now we give a different type of condition
for $L^\infty$-extension to hold in the critical case $m=n$
from those given in Section \ref{Lpext},
but by using Lemma \ref{Lemma:basic} again.
\begin{cor}\label{Prop:FS}
Let $X\subset\Rn$ be an open bounded set,
let $A=\{0\}$ be the origin,
and let $P$ be a partial differential operator of
order $m=n$.
Suppose that $u\in L^\infty(X)$ satisfies
$Pu=0 \textrm{ on } X\backslash A$.
Then either 
$u$ is a non-zero constant multiple of a fundamental solution of $P$
or $Pu=0$ on $X$.
In other words, $L^\infty$--extension fails if and only if
$P$ has a bounded fundamental solution in $X$
\end{cor}
\begin{proof}[Proof of Corollary \ref{Prop:FS}.]
Assumption that $Pu=0$ on $X\backslash A$
implies expression \eqref{EQ:structure}
by the structure theorems for distributions.
Then by Lemma \ref{Lemma:basic} with $p=\infty$ and $m=n$, 
we have that $Pu=c\delta$ with a constant $c$.
Then the conclusion is straightforward.
\end{proof}
For example, fundamental solutions of the Laplace operator
$\Delta$ in $\R^2$ all have logarithmic growth at the origin,
hence the $L^\infty$--extension holds by Corollary \ref{Prop:FS}.
From this observation and Corollary \ref{cor:FS-stronger} 
with $p=\infty$,
we obtain the following well-known 
(at least for $n=2$ when $A$ is a point) result:
\begin{cor}\label{harm}
Let $X\subset\Rn$, $n\geq 2$, 
be an open bounded set and let $A$ be a smooth submanifold
of $X$ of codimension 2.
Every bounded harmonic function on $X\backslash A$
is harmonic on $X$ if it is bounded in a neighbourhood of $A$. 
\end{cor}
The case of a smooth manifold $A$ of dimension greater than
zero will follow from the arguments of the following section.

Finally, let us give an additional clarification of
conditions \eqref{EQ:cond-gamma-cora} and
\eqref{EQ:cond-gamma-cor} in the critical case of $L^\infty(X)$
and $m=n$:
\begin{rem}\label{rem:weakconv}
Let us take $p=r=\infty$ in Theorem \ref{cor:all-U-Lp}.
Condition \eqref{EQ:cond-gamma-cora} requires that for some $\phi$,
$\epsilon_j$, and
$Q$ we have
\begin{equation}\label{EQ:cond-gamma-sp}
 \epsilon_j^{n+d}
 Q (aPu*\phi_{\epsilon_j})\to 0 \textrm{ as } 
 \epsilon_j\to 0+,
\end{equation} 
with the convergence in the $L^\infty$--norm. We claim 
that it actually always holds in
the weak$^*$ topology of $L^\infty(X)$
(for a subsequence which we may in turn choose to be the
sequence $\{\epsilon_j\}$ in
\eqref{EQ:cond-gamma-cora}). 
Indeed, since by Proposition \ref{cor:all-U-Lp2}
the norm $|| Q (aPu*\phi_{\epsilon_j})||_
{L^\infty}$ is of order $\epsilon^{-d-m}$ for
$\epsilon\to 0+$, it follows that the family
in \eqref{EQ:cond-gamma-sp} 
is uniformly bounded in $L^\infty(X)$, provided that
$m\leq n$. Since the
balls are compact in the weak$^*$ topology of $L^\infty$,
it follows that there is a subsequence of $\epsilon_j$,
for which the corresponding subsequence of the family
in \eqref{EQ:cond-gamma-sp}
converges in the weak$^*$ topology. Since the distributional
limit is zero, it follows that the subsequence also converges
to zero in the weak$^*$ topology of $L^\infty(X)$.
\end{rem}

\section{Concluding remarks}
\label{Sec:generalisations}
The results in this paper can be
extended to the case when the removable set
$A$ under study is not a point but a smooth submanifold of $X$
of codimension $k$.
Let us take local coordinates $x=(x_1,\dots,x_n)$ on $X$ so that 
$$A=\{(x',x''):x''=0\},$$
where we denote $x'=(x_1,\dots,x_{n-k})$, $x''=(x_{n-k+1},\dots,x_n)$.
Let us also use the notation
$$D'=(D_1,\dots,D_{n-k}),\quad D''=(D_{n-k+1},\dots,D_n),$$
where $D_j=\partial/\partial x_j$ $(j=1,\dots,n)$.
Then, by the structure theorem, distribution $f$ on $\R^n$ 
with support in $A$ is expressed as 
$$f=S(x',D'')\delta(x'')$$
for a differential operator $S(x',D'')$ with distribution coefficients
in $x'$ and independent of $D'$, 
where $\delta(x'')$ denotes the delta function in $x''$.

On account of these observation, 
for every $a(x)\in C^\infty_0(X)$
and $\psi(x')\in C^\infty_0(\R^{n-k})$,
condition
$$Pu=0 \textrm{ on } X\backslash A$$
implies
\[
\langle aPu,\psi(x')\otimes\cdot\rangle =S(D'')\delta(x'')
\text { on $\R^k$},
\]
which is a starting expression instead of \eqref{str-thm}.
Then all the results in Section 2 and Section 3 for
the case $A=\{0\}$ can be easily generalised to the case $A=\{(x',x''):x''=0\}$
by replacing the dimension $n$ by the codimension $k$ of the set $A$.
For example, Theorem \ref{cor:all-U-Lp}
is true in such a general case
if condition \eqref{EQ:cond-gamma-cora} (similarly \eqref{EQ:cond-gamma-cor})
is replaced by
$$
\epsilon_j^{k(1-1/r)+d}||
 Q(a Pu*\phi_{\epsilon_j})(0,x''))
 ||_{L^r(\R^{k})}\to 0 \textrm{ as } \epsilon_j\to 0+.
$$
Proposition \ref{cor:all-U-Lp2} is also true if we replace
 $|| Q(a Pu*\phi_{\epsilon}) ||_{L^p(\R^{n})}$
in \eqref{O-beh} and \eqref{o-beh}
by
 $|| Q(a Pu*\phi_{\epsilon})(0,x''))||_{L^p(\R^{k})}$.
Hence Corollary \ref{cor:FS-stronger} holds
when we replace $n(1-1/p)$ by $k(1-1/p)$.

It is also easy to see that all the results in this paper,
including the discussion above, are still true in the case
when $P$ is a matrix of partial differential operators of order $m$,
and $u$ is a vector of functions in $L^p(X)$.
In this way, we can also discuss the case when
$P=P(z,D_z)$ is a partial differential operators of order
$m$ on the complex space $\Cpx^n$.
In fact, equation $Pu=0$ on $\Cpx^n$ can be written as a system of
equations on $\R^{2n}$ if we regard the complex variable $z$ as
a pair of real variables $(x,y)$ and consider the real and 
imaginary part of the equation independently.

Furthermore, generalisation of
the results to the case of higher dimensional $A$ as above implies
that the critical case of the $L^\infty$-extension
problem corresponds to the analytic extension 
over smooth real surfaces $A$ in $\Cpx^n$ 
of real codimension one. 
Indeed, the standard Riemann's extension 
problem is the subcritical case of the $L^\infty$--extension
problem, while the extension problem over curves in
$\Cpx$ is the critical case since the codimension of a curve
equals to the order of the Cauchy--Riemann operator,
and it can be easily seen that the analytic extension
fails in this case (in our terminology the $L^\infty$--extension
property fails for operator
$\overline{\partial}$ in the critical case).
For example, take an analytic function in
the unit disc in $\Cpx$, cut the disc by a straight line and then
shift the graph of the function in one of
the half-discs along this line.

In particular, Corollary \ref{Prop:FS} corresponds to
the analogue of the
Riemann's extension theorem for analytic functions over
real curves in $\Cpx$. 
This corresponds to the existence of bounded fundamental 
solutions for the Cauchy--Riemann operator restricted to
smooth real curves in the complex plane.



\begin{thebibliography}{99}

\bibitem{Bo} S.~Bochner, 
Weak solutions of linear partial differential equations,
{\em J. Math. Pures Appl.} {\bf 35} (1956), 193--202. 

\bibitem{Jo} F.~John, 
{\em Plane waves and spherical means applied to 
partial differential equations}. Reprint of the 1955 original. 
Dover Publications, Inc., Mineola, NY, 2004.

\bibitem{HP} R.~Harvey and J.~Polking, 
Removable singularities of solutions of linear partial 
differential equations,
{\em Acta Math.} {\bf 125} (1970),  39--56. 

\bibitem{Ho} L.~H\"ormander, 
{\em Linear partial differential operators}. Springer-Verlag,
1963. 

\bibitem{Li} W.~Littman, 
Polar sets and removable singularities of partial 
differential equations,
{\em Ark. Mat.} {\bf 7} (1967), 1--9. 

\bibitem{P} A.~V.~Pokrovskii, Removable singularities 
of weak solutions of linear partial differential equations, 
{\em Math. Notes} {\bf 77} (2005), 539--545.

\bibitem{Po} J.~Polking, A survey of removable singularities. 
Seminar on nonlinear partial differential equations 
(Berkeley, Calif., 1983), 261--292,
Math. Sci. Res. Inst. Publ., 2, Springer, New York, 1984. 

\bibitem{Sh} N.~Shimakura, 
{\em Partial differential operators of elliptic type}. 
Translations of Mathematical Monographs, 99. 
American Mathematical Society, Providence, RI, 1992. 

\bibitem{S1} M.~Sugimoto,
A weak extension theorem for inhomogeneous differential equations,
{\em Forum Math.} {\bf 13} (2001), 323--334.

\bibitem{S2} M.~Sugimoto,
A new aspect of the $L\sp p$-extension problem for inhomogeneous differential
equations, Modern trends in pseudo-differential operators, 153--159,
{\em Oper. Theory Adv. Appl.}, {\bf 172}, Birkh\"auser, Basel, 2007. 


\bibitem{SU} M.~Sugimoto and M.~Uchida, 
A generalization of Bochner's extension theorem and 
its application, {\em Ark. Mat.} {\bf 38}  (2000),  399--409.





 \end{thebibliography}
\end{document}